\documentclass[a4paper,10pt]{amsart}
\usepackage[utf8]{inputenc}
\usepackage[T1]{fontenc}

\title[Virtual volumes of strata of meromorphic differentials ]{Virtual volumes of strata of meromorphic differentials with simple poles}
\date{\today}
\author{Adrien Sauvaget}
\address{Institut Mathématique de Toulouse; UMR 5219; Université de Toulouse;
CNRS; UPS, F-31062 Toulouse Cedex 9, France}
\email{adrien.sauvaget@math.cnrs.fr}

\usepackage{amsmath,amssymb,amsthm,mathtools} 
\usepackage{color}
\usepackage{graphicx}
\usepackage[all]{xy}

\usepackage{hyperref}

\usepackage{enumitem}
\setlist{leftmargin=5.5mm}

\usepackage{caption}
\theoremstyle{definition}
\newtheorem{definition}{Definition}[section]
\newtheorem{notation}[definition]{Notation}

\newtheorem{remark}[definition]{Remark}

\theoremstyle{plain}

\newtheorem{theorem}[definition]{Theorem}
\newtheorem{proposition}[definition]{Proposition}

\newtheorem{lemma}[definition]{Lemma}
\newtheorem{corollary}[definition]{Corollary}

\def\oM{\overline{\mathcal{M}}}

\def\oH{\overline{\mathcal{H}}}

\def\H{{\mathcal{H}}}

\def\M{{\mathcal{M}}}

\def\F{\mathcal{F}}

\def\RR{\mathbb{R}}
\def\CC{\mathbb{C}}
\def\ZZ{\mathbb{Z}}
\def\NN{\mathbb{N}}
\def\QQ{\mathbb{Q}}
\def\PP{\mathbb{P}}

\def\LL{\mathbb{L}}

\def\a{\underline{a}}
\def\b{\underline{b}}

\def\sskip{\vspace{4pt}}

\begin{document}

\maketitle

\begin{abstract}
We work over strata  of meromorphic differentials with poles of order $1$, and on affine subspaces defined by linear conditions on the residues. We propose a definition of the volume of these objects as the integral of a tautological class on the projectivization of the stratum. By previous work with Chen--Möller--Zagier, this definition agrees with the Masur--Veech volumes in the holomorphic case.  We show that these algebraic constants can be computed by induction on the genus and number of singularities. Besides, for strata with a single zero, we prove that the generating series of these volumes is a solution of an integrable system associated with the PDE: $u_tu_{xx}=u_tu_x+u_t - 1$. 
\end{abstract}

\setcounter{tocdepth}{1}
\tableofcontents

\section{Introduction}

\subsection{Strata of meromorphic differentials}
Let $g, n$ and $m$ be non-negative integers satisfying the stability condition $2g-2+n+m>0$, and let $\mu=(k_1,\ldots,k_n)$ be a vector in $\NN^n$ satisfying
$$
|\mu| \overset{\rm def}{=} \sum_{i=1}^n k_i = 2g-2 +n+ m. 
$$ 
We work over the generalized Hodge bundle $\oH_{g,n,m}\to \oM_{g,n+m}$ of meromorphic differentials with poles of order at most $1$ at the $m$ last markings. Let $\H(\mu,m)\subset \oH_{g,n,m}$ be the locus of differentials $(C,x_1,\ldots,x_{n+m}, \eta)$ such that $C$ is smooth and ${\rm ord}_{x_i}\eta = k_i-1$ for all $i\in \{1,\ldots,n\}$, while ${\rm ord}_{x_i}\eta=-1$ for all $i\in \{n+1,\ldots,n+m\}$. Let $\rho=(\rho_1,\ldots,\rho_\ell)\in \ZZ_{>0}^\ell$ such that $|\rho|=m$, then we denote 
$$
R(\rho)=\left\{(r_{n+1},\ldots,r_{n+m})\in \CC^m, \text{ s.t. }  \sum_{i=\rho_1+\ldots+\rho_{k-1}}^{\rho_1+\ldots+\rho_{k}} r_i = 0   \text{ for all $1\leq k\leq \ell$}\right\}
$$ The {\em stratum of meromorphic differentials with residue conditions}  $\H_g(\mu,\rho)\subset \H(\mu,m)$ is the locus of differentials such that the vector of residues at the $m$ poles  lies in $R(\rho)$. 
This space is a smooth complex orbifold of dimension 
$$
d(\mu,\rho)= |\mu|-\ell(\rho)+1.
$$
Note that it is empty if at least one $\rho_i$ is $1$.  We denote by $\oH_g(\mu,\rho)$ the closure of $\H_g(\mu,\rho)$ in $\H_{g,n}$ and by  $\PP\H_g(\mu,\rho)$ and $\PP\oH_g(\mu,\rho)$ the projectivizations. 

\subsection{Masur--Veech volumes} The stratum $\H_g(\mu,\rho)$ is canonically endowed with a linear structure over $\ZZ$, i.e. there exists an atlas of (period) charts  such that the transition functions are linear functions defined by matrices with coefficients in $\ZZ$. These period coordinates allow for the definition of: an action of  ${\rm GL}(2,\RR)$ (by simultaneous action on periods), and a canonical measure $\widetilde{\nu}_{MV}$, expressed as the euclidean measure in period coordinates. This measure is ${\rm GL}(2,\RR)$-invariant by construction. 

\sskip

For {\em strata of holomorphic differentials}, i.e. when $\rho=()$, the area of a differential is defined as 
$$
{\rm area}(C,\eta) =  \frac{i}{2} \int_{C} \eta\wedge \overline{\eta}
$$
 is finite. By construction, the area is ${\rm SL}(2,\RR)$-invariant, and quadratic with respect to scaling by a factor in $\RR_{>0}$.  We denote by $\H_g(\mu)_{1}$ the subspace of $\H_g(\mu)$ of differentials of area $1$. This measure can be decomposed as $\nu_{MV}\wedge d({\rm area})$. The restriction of $\nu_{MV}$ to $\H_g(\mu)_{1}$ is the {\rm Masur--Veech measure}.  The {\em Masur--Veech volume} of the stratum is defined 
\begin{equation}{\rm Vol}(\mu) =  \nu_{MV}(\H_g(\mu)_{1}).
\end{equation}
By result of Masur and Veech, this quantity is finite~\cite{Mas,Vee}. In fact, the action of ${\rm SL}(2,\RR)$ on $\H_1(\mu)$ is ergodic and preserves $\nu_{MV}$.  Hence, the this measure  plays an important role in the description of the properties of this dynamical system. Besides, this measure can be used to describe the growth of the number of geodiscs on a generic differential in $\H_g(\mu)_{1}$ by work of Veech, Masur, Eskin, and Zorich among others~\cite{Vee,EskMas,EskMasZor}.   
 
 \sskip
 
One cannot expect to extend this picture to the strictly meromorphic setting. Indeed, if $f:\H_g(\mu,\rho)\to \RR_{>0}$ is a ${\rm SL}(2,\RR)$-invariant function  which is quadratic with respect to scaling, then the associated measure has infinite volume. Hence, we propose a purely algebraic definition of this volume via intersection theory. For a reader interested in the dynamics of translation surfaces (with boundaries in the meromorphic case), the natural idea is to  replace $\H_g(\mu,\rho)$ by the boundary locus $\oH_g(\mu,\rho)_0\subset \oH_g(\mu,\rho)$ of differentials with no trivial residues. Indeed, this locus is ${\rm SL}(2,\RR)$ invariant so we can look for non-negative ${\rm SL}(2,\RR)$-invariant distributions concentrated along this locus\footnote{Indeed, a ${\rm SL}(2,\RR)$-invariant distribution on $\H_g(\mu,\rho)$ induces a ${\rm SL}(2,\RR)$-invariant distribution on the space of residues. Such distribution has finite integral only if it is supported at 0.}. The {\em expected} co-dimension  of $\oH_g(\mu,\rho)_0$  is ${\rm dim}(R(\rho))$, but $\oH_g(\mu,\rho)_0$ contains components of smaller co-dimension, hence the motivation for a {\em virtual} volume.

\subsection{Volumes of strata as intersection numbers} We consider the following cohomology classes:
\begin{itemize}
    \item $\xi=c_1(\mathcal{O}(1))\in H^2(\PP\oH_g(\mu,\rho),\QQ)$;
    \item $\psi_i=c_1(\LL_i)\in H^2(\oM_{g,n},\QQ),$ where $\LL_i$ is the co-tangent line at the $i$-th marking for all $i\in \{1,\ldots,n\}$.
\end{itemize}
For all $i\in \{1,\ldots,n\}$, we denote 
\begin{equation}
    a(\mu,\rho)_i=  \int_{\PP\oH_g(\mu,\rho)} \xi^{d(\mu,\rho)-n+1} \prod_{j\neq i} k_j\psi_j
\end{equation}

\begin{proposition}[see Section~\ref{ssec:cohomology}] \label{prop:independence} The value $a(\mu,\rho)_i$ is independent of $i$ and thus is denoted by $a(\mu,\rho)$. 
\end{proposition}
\begin{definition}
    We define the {\em virtual volume} (or {\em volume} for short) of $\PP\oH_g(\mu,\rho)$ as
    \begin{equation}
     {\rm Vol}(\mu,\rho)\coloneq -\frac{2(2i\pi)^{2g}}{(|\mu|-1)! \prod_{i=1}^n k_i } a(\mu,\rho).  
    \end{equation}
\end{definition}
In the holomorphic case, the main results of~\cite{CMSZ} shows that this volume is equal to the Masur--Veech volume of the stratum defined in the previous section.  In general, the use of the world ``volume'' is abusive as ${\rm Vol}(\mu,\rho)$ is not defined by a measure. However, we will prove that ${\rm Vol}(\mu,\rho)$ is positive. In fact we will show that it is the volume of a union of boundary components of  $\PP\oH_g(\mu,\rho)$ of co-dimension ${\rm dim}(R(\rho))$, and along which the residues are trivial. 

\begin{remark} There are several choices for this representation of ${\rm Vol}(\mu,\rho)$, and no ``canonical'' one unless  $\rho=(2)$. In this case, ${\rm Vol}(\mu,\rho)$ is the volume of the principal boundary of Eskin-Masur-Zorich for cylinder degenerations~\cite{EskMasZor}. In particular the {\em area Siegel-Veech constant} is given by
$$
c_0(\mu)= \frac{1}{4\pi^2} \frac{{\rm Vol}(\mu,(2))}{{\rm Vol}(\mu,())}.
$$
The interpretation of the next volumes in terms of number of long geodesics would be desirable, but it is not clear how to formulate a conjectural answer to this problem. Indeed, work in progress of Aulicino--Masur--Pan--Su indicates that the large length behaviour is not quadratic \cite{oberwolfachMasur}. Yet, the present definition of volumes could still be used to  compute the asymptotic proportion of geodesics in a given configuration. 
\end{remark}

\begin{remark} We do not include it here, but one should also be able to generalize the present definition to moduli space of meromorphic $k$-differentials with poles of order at most $k$ using the (conjectural) expression of these volumes as intersection numbers in~\cite{SauFlat}.
\end{remark}

\subsection{Combinatorics of virtual volumes} 

We introduce the following polynomials in  $\QQ[t_2,t_3\ldots]$
\begin{eqnarray}
\a(\mu) \coloneqq \sum_{\ell\geq 0}  \sum_{\rho=(\rho_1,\ldots,\rho_{\ell})} \frac{1}{\ell!} \cdot a(\mu,\rho) \cdot \prod_{i=2}^\ell t_{\rho_i} \in \QQ[t_2,t_3\ldots].
\end{eqnarray}
These polynomials can be considered as refinement of Masur--Veech volumes  in $t$-variables. Indeed, the constant coefficient of this polynomial is $a(\mu,())$, i.e. (up to a constant) the classical Masur--Veech volume in a genus determined by $\mu$ (if this genus is a half-integer then this is zero), while the next coefficients are volumes over strata in smaller genera. 

\sskip

Our first theorems show that the induction formulas of~\cite{SauGAFA,CMSZ} for Masur--Veech volumes, hold without modification once we have refined the initial input. To state these results, we use the classical notation
$$
\mathcal{S}(z)=\frac{{\rm sinh}(z/2)}{z/2}= \sum_{g\in \NN/2} b_g z^{2g} = 1 -\frac{1}{24} z^2 +\frac{7}{5760} z^4  -\frac{ 31}{967680} z^6+\ldots 
$$ 
The summation is over non-negative half-integers although $b_{g+1/2}=0$ for an integral $g$. Then we define 
\begin{eqnarray*}
    \b_g&=& [z^{2g}] \, {\rm exp}\left(z^2 t_2 + z^3 t_3+ \ldots \right) \cdot \mathcal{S}(z)^{-1}\\
     &=& b_{g} + \sum_{\rho\neq ()} b_{g-|\rho|/2} \frac{\prod t_{\rho_i}}{|{\rm Aut}(\rho)|}.
\end{eqnarray*}
Here, the notation $[\star]$ stands for ``the coefficient of $\star$ in the following formal series, and the sum is over non-trivial partitions. Here are the first $\b$-polynomials:
$$
\begin{aligned}
\b_0 &= 1     \qquad & \b_{1/2} &= 0 \\
\b_1 &= -\frac{1}{24} + t_2     \qquad & \b_{3/2} &= t_3 \\
\b_2 &= \frac{7}{5760} - \frac{1}{24} t_2+ t_4+ \frac{t_2^2}{2}  \qquad & \b_{5/2} &= -\frac{1}{24} t_3 +  t_5+t_3t_2\\
\b_3 &= -\frac{ 31}{967680} + \frac{7}{5760} t_2 - \frac{1}{24}\left(t_4+\frac{t_2^2}{2}\right) & \ldots
\end{aligned}
$$

\begin{theorem}\label{thm1}
If we set $\mathcal{F}(z)=1+\sum_{k\geq 1} k \a((k)) z^{k+1},$ then for all half-positive integers $g$, the following identity holds
\begin{equation}\label{for:thm1}
\b_{2g} = [z^{2g}] \frac{\F^{2g}}{(2g)!}.
\end{equation}
\end{theorem}
This  theorem determines the volumes of all minimal strata (i.e. when $n=1$). Here are the first relations:
$$\begin{aligned}
0! \b_1 &= \a((1))     & 1! \b_{3/2} &= \a((2)) \\
2! \b_2 &= \a((3)) + \frac{1}{2} \a((1))^2  & 3! \b_{5/2} &=  \a((4)) + 2\a((2))\a((1)) \\
4! \b_3 &= \a((5)) + 3\a((3))\a((1)) \!\!\!\! \! &  + \frac{4}{2} \a((2))^2&+ \frac{1}{6} \a((1))^3  \ldots&
\end{aligned}
$$

 To state the next induction formula, we work over another ring of polynomials in infinitely many formal variables $R=\QQ[h_1,h_2\ldots]$. Then, for all $i\in \NN^*$ we set
$$
\mathcal{H}_{{\{i\}}}(z_i)= z_i^{-1} \sum_{k\geq 1} h_k z_i^k \in z_i^{-1}\cdot R[[z_{i}]].$$
For all sets  $I\subset \NN^*$ of size at least two, we  define the formal series $\H_I$ in  $R[[z_{i}]]_{i\in I}$ inductively by
\begin{eqnarray*}
\mathcal{H}_{\{i,j\}} &=& \frac{z_i \H_{\{i\}}'-z_j \H_{\{j\}}'}{ \H_{\{j\}}-\H_{\{i\}}}  -1,  \text{ and } \\
\mathcal{H}_{I} &=&  \frac{1}{2(|I|-1)}\underset{I',I''\neq \emptyset}{\sum_{I'\sqcup I''=I}} D_2(\H_{I'},\H_{I''}), \text{ if $|I|>2$},
\end{eqnarray*}
where $D_2(f,g)=\sum_{k_1,k_2\geq 1} \left([z_1^{k_1}z_2^{k_2}] \H_{\{1,2\}}\right) \frac{\partial f}{\partial k_1}  \frac{\partial g}{\partial k_2}$.
\begin{theorem}\label{thm2}
We define the ring morphism ${\rm ev}\colon R \to \QQ[t_2,t_3,\ldots]$ on generators by $h_k\mapsto 2k \a((k))$. Then for all $\mu$ we have 
$$
\a(\mu) =  \frac{1}{2|\mu|} \cdot {\rm ev}\left( [z_1^{k_1}\ldots z_n^{k_n}] \H_{\{1,\ldots,n\}}\right).
$$
\end{theorem}
This second theorem determines all $\a(\mu)$ in terms of the $\a((k))$. The first relations are as follows 
$$\begin{aligned}
\a(2,2) &= 3\a((3)) + \frac{1}{2} \a((1))^2     & \a(3,2) &= 4 \a((4)) + 2\a((2))\a((1)) \\
\a(4,2) &= 5\a((5)) + \frac{8}{2} \a((2))^2  & \a(3,3) &=  5\a((5)) + \frac{8}{2} \a((2))^2 +  \frac{2}{6} \a((1))^3  \\
\a(2,2,2) &= 3\a((3,2)) + \a((1))\a((2,1)) &\ldots \end{aligned}
$$

\begin{remark}
	These two theorems also allow for the description of the large genus asymptotic behavior of these volumes as in the holomorphic case~\cite{EskZor}. Indeed, we claim that for a fixed $\rho$, there exists a constant $C>0$ such that 
	$$
	\left|(2g)^{-|\rho|}{\rm Vol}(\mu,\rho) - \frac{4}{\prod_{i=1}^n k_i}\right|< C
	$$
	for all $\mu$. We do not prove this statement to keep the presentation concise, but it follows from the strategy of~\cite{SauIRMN} (based on the results of \cite{Agg}). 
\end{remark}

\subsection{An integrable hierarchy for strata with a single zero} 

Here we present a different set of relations between $\a$-polynomials that correspond to ``removing a pole'' (instead of diminishing $g$ or $n$). While the previous set of relations generalizes~\cite{SauGAFA,CMSZ} for the constant coefficients, this relations of this section can be seen as generalization of~\cite{BurRoscount,ChePra} in genus 0. Here we restrict ourselves to the case $n=1$ to construct an integrable hierarchy.
\begin{theorem}\label{thm3}
If we set $u(x)= \sum_{k\geq 1}  a((k)) x^{1-k}$, then 
\begin{eqnarray}\label{for:thm31}
\partial_{t_2} u &=& \frac{1}{1-u''+u'}, \text{ and }\\ \label{for:thm32}
 \partial_{t_m} u &=&  \partial_{t_2} u \cdot \partial_{t_{m-1}}(-u'+u) \text{ for all $m\geq 3$}.
\end{eqnarray}
\end{theorem}
This result provides a new type of interplay between integrable systems and intersection theory of moduli spaces of curves. Indeed, here the parameter $t_m$ does not account for an insertion at a specific marking, but corresponds to a constraint (the residue condition) at a group of markings. 

\begin{remark} the operators $\partial_x, \partial_{t_2},  \partial_{t_3},\ldots$ commute but they are derivations on the algebra $\QQ[[u,u',u'',\ldots]]$ while the most classical notion of integrable hierarchy requires to work over $\QQ[u',u'',\ldots][[u]]$. However, we obtain a more restrictive statement if we observe that all these operators can be defined formally as derivations on the algebra $
 \QQ[\partial_{t_2}u, u, u',u'',\ldots]$. 
 \end{remark}
 
\subsection{Acknowledgements} I am thankful to A. Zorich for suggesting this problem, and G. Forni for useful conversations on the topic. I am also grateful to A. Wright who explained to me some years ago that ${\rm SL}(2,\RR)$-invariant measures on moduli space of meromorphic differentials are necessarily infinite, thus preventing me from obvious pitfalls.

\section{Moduli spaces of differentials and their compactification}

Here, we recall some of the results about moduli spaces of (meromorphic) differentials and their compactification from~\cite{BCGGM, SauGT}. In this section we chose $g,n,m$ as in the introduction. We also fix a pair $(\mu,\rho)$, but we only assume that $|\mu|\leq 2g-2+n+m$. Then the space $\oH_g(\mu,\rho)$ is the closure of the locus of differentials with order at $x_i$ greater than $k_i$. If $|\mu|=2g-2+n+m$, then we say that $(\mu,\rho)$ is {\em complete}.

\subsection{Back-bone graphs}

We fix a stable pair $(g,n)$, and a stable graph of genus $g$ with $n$ labeled legs
$$
\Gamma=(V,H,g:H\to \NN,
\iota:H\to H,\phi:H\to V,H^\iota\simeq \{1,\ldots, n\}) 
$$
(we refer to~\cite{SauGT} for the definition of stable graph). 
\begin{definition}
    A {\em twist} on $\Gamma$ is a function $b\colon H\to \ZZ$ satisfying the following constraints.
    \begin{itemize}
        \item For all edges $(h,h')$, we have $b(h)+b(h')=0$.
        \item If $(h_1,h_1')$ $(h_2,h_2')$ are edges connecting the vertices $v$ and $v'$, then $b(h_1)\geq 0 \Rightarrow b(h_2)\geq 0$. In which case  we denote $b(v)\geq b(v')$.
        \item The relation $\geq$ is a partial order.
        \item For all vertices $v$, we have 
        $$\sum_{h\mapsto v}b(h)\leq 2g(v)-2+n(v).$$ 
    \end{itemize}
    The pair $(\Gamma,b)$ is called a {\em twisted graph}. We say that a vertex is {\em complete} if the last inequality is an equality. 
\end{definition}
\begin{definition}
    A {\em bi-colored graph} is a twisted graph $(\Gamma,b)$ such that there exists a non-trivial partition of the set of vertices $V=V_0\sqcup V_{-1}$ satisfying: all edges $(h,h')$ are between a vertex $v\in V_0$ and a vertex $v'\in V_{-1}$, and $b(h)>0$. If such a partition exists, one can check that it is unique.
\end{definition}

\begin{notation}
    If $(\Gamma,b)$ is a bi-colored graph, then we denote
    \begin{itemize}
        \item $b(v)\in \ZZ^{n(v)}$ the vector of twists at half-edges incident to $v$,
        \item $b(e)=|b(h)|=|b(h')|$ if $e=(h,h')$ is an edge, and the {\em multiplicity} of $(\Gamma,b)$ is defined as
        $$
        m(\Gamma,b)=\prod_{e} b(e).
        $$
    \end{itemize}
\end{notation}

\begin{definition} A {\em rational back-bone graph} (or simply {\em back-bone graph} in the text) is a bi-colored graph $(\Gamma,b)$ such that:
\begin{itemize} \item the graph $\Gamma$ is of compact type; 
\item there is a unique vertex of in $V_{-1}$ which is of genus 0;
\end{itemize}
We say that it is {\em complete} if each vertex in $V_0$ is complete. 
\end{definition}

\begin{remark}
If $(\Gamma,b)$ is a complete back-bone graph, then $b$ is completely determined by $\Gamma$, or  conversely, the function $b$ determines the genus assignment for vertices. We will often use these facts to re-organize various sums over back-bone graphs.
\end{remark}

\subsection{Segre classes of strata}  We consider the vector bundle $\oH_{g,n,\rho}\subset \oH_{g,n,m}$ of differentials with residues lying in $R(\rho)$. It is a vector bundle of rank $g+m-\ell(\rho)$, and we have the exact sequence
$$
\mathcal{O}\to \oH_{g,n+m}\to \oH_{g,n,\rho} \overset{\rm res}{\to} \CC^{m-\ell(\rho)} \to \mathcal{O},
$$ 
i.e. the kernel of the residue map is the Hodge bundle, the vector bundle of holomorphic differentials. In particular we have 
$$
c_*\left(\oH_{g,n,\rho}\right)=c_*\left(\oH_{g,n,\rho}\right)=1+\lambda_1+\ldots+\lambda_g
$$
\begin{lemma}\label{lem:vanish} If $p\colon \PP\oH_{g,n,\rho}\to \oM_{g,n+m}$ stands for the forgetful morphism of the differential then 
$$p_*\xi^{2g-1+m-\ell(\rho)+k}=\left\{ \begin{array}{cl} 0 & \text{if $k>0$,} \\ (-1)^{g} {\lambda_g}.\end{array} \right.$$ 
\end{lemma}
\begin{proof}
We recall that $p_*\xi^{2g-1+m-\ell(\rho)+k}$ is the Segre class $s_{g+k}\left(\oH_{g,n,\rho}\right)$. By Mumford's identity~\cite{Mum} we have
$$
s_*\left(\oH_{g,n,\rho}\right) = c_*\left(\oH_{g,n,\rho}\right)^{-1} =c_*\left(\oH_{g,n+m}\right)^{-1} = 1-\lambda_1 +\ldots+ (-1)^g\lambda_g.
$$
\end{proof}
In particular, using the projection formula along $p$, we have the immediate corollary. 
\begin{corollary} For all $\alpha\in H^*(\PP\oH_{g,n,\rho},\QQ)$ we have $\int\alpha\cdot \xi^{2g-1+m-\ell(\rho)+k}=0$ if $k> 0$.
\end{corollary}

\subsection{Boundary components}  Let $\Gamma$ be a stable graph of genus $g$ with $n+m$ legs. We recall that $\Gamma$ determines a boundary component of $\oM_{g,n+m}$ i.e. a moduli space and a finite morphism
$$
\zeta_{\Gamma}\colon \oM_{\Gamma}\coloneqq \prod_{v\in V} \oM_{g(v),n(v)} \to \oM_{g,n+m}.
$$

\begin{definition}
A twist $b$ on $\Gamma$ is said {\em compatible with $(\mu,\rho)$} if it satisfies the conditons at legs: $b(i)=k_i$ for $1\leq i\leq n$, while $b(i)=0$ if $n<i\leq n+m$. 
 \end{definition}

We denote by ${\rm Bic}_g(\mu,\rho)$ the set of bi-colored graphs of genus $g$ and compatible with $(\mu,\rho)$. If $\Gamma$ is a twisted graph in ${\rm Bic}_g(\mu,\rho)$, then  each vertex $v$ in $V^0$ defines a pair $(\mu(v),\rho(v))$:  the first one consists of the zero orders (either at legs with indices at most $n$ or at edges to level $-1$ vertices) while the second one is determined by the vector space 
$$
R(\rho) \cap \{(r_{n+1}, \ldots,r_{n+m}), \text{ s.t.  $r_i=0$ if $n+i$ is not incident to $v$}\}.
$$
Then $(\Gamma,b)$ determines a moduli space and a morphism
$$
\zeta_{\Gamma,b} \colon \oH(\Gamma,b,\rho) \coloneqq \oM_{-1} \times \prod_{v\in V^0} \oH_g(\mu(v),\rho(v)) \to \oH_g(\mu,\rho), 
$$
where $\oM_{-1}$ is a subspace of $\prod_{v\in V_{-1}} \oM_{g(v),n(v)}$ described in~\cite[Section~4.2]{SauGT}. As we do not need this definition in general, we  will only give it further for certain back-bone graphs. The morphism $\zeta_{\Gamma,b}$ is defined by associating to a point in $\oH(\Gamma,b,\rho)$ the curve defined by the gluing morphism $\zeta_{\Gamma}$, and the differential that vanishes along components associated with vertices in $V^{-1}$ and prescribed by the value in $\oH_g(\mu(v),\rho(v))$ otherwise. This morphism is finite of degree $|{\rm Aut}(\Gamma,b)|$. 

\begin{definition}
Let $(\Gamma,b)$ be a back-bone graph compatible with $(\mu,\rho)$. We say that $(\Gamma, b)$ is {\em $\rho$-simple} if for all $1\leq k\leq \ell(\rho)$ we have: the legs with label $i$ satisfying $$n+\rho_1+\ldots + \rho_{k-1} < i \leq  n+\ldots + \rho_k$$ are incident to the same vertex in $V_0$.  We say that it is {\em almost $\rho$-simple} if the above condition holds for $1\leq k<\ell(\rho)$, but if $k=\ell(\rho)$  then either $\rho_k=2$ and the last two legs sit on the vertex in $V_{-1}$, or otherwise the last leg sits on the vertex in $V_{-1}$ while the legs with label $i$ satisfying
 $$n+\rho_1+\ldots + \rho_{\ell-1} < i <  n+m$$ are incident to the same vertex in $V_0$.
 \end{definition}

\begin{lemma}\label{lem:vanishsbb} Let us denote $d=2g-1+m-\ell(\rho)$.
 \begin{enumerate} 
	\item If $(\Gamma,b)$ is a back-bone graph compatible with $(\mu,\rho)$, then the class $$p_*(\xi^{d}\zeta_{\Gamma,b\, *}[\PP\oH(\Gamma,b,\rho)])$$ vanishes unless $(\Gamma,b)$ is a simple back-bone graph. 
	
	\item If we assume also that $m>0$, and the last leg sits on a vertex in $V_0$, then $$p_*(\xi^{d-1}\zeta_{\Gamma,b\, *}[\PP\oH(\Gamma,b,\rho)])$$ vanishes unless $(\Gamma,b)$ is an almost simple back-bone graph.
	\end{enumerate}
\end{lemma}

\begin{proof}
	Here we decompose the Segre class of a product of cones
	\begin{eqnarray*}
	p_*(\xi^{D}\zeta_{\Gamma,b\, *}[\PP\oH(\Gamma,b,\rho)]) =  \underset{|\underline{d}|=D-|V_0|+1}{\sum_{\underline{d}=(d_v) \in \NN^{V_0} }} \zeta_{\Gamma*}\left( \bigotimes_{v \in V_0} p_{v*}\xi^{d_v} \right)
	\end{eqnarray*}
	where the notation $p_v$ stands for the forgetful morphism $$\PP\oH_{g(v)}(\mu(v),\rho(v))\to \oM_{g(v),n(v)+m(v)}.$$ This class vanishes if $d_v$ is greater than $2g(v)+|\rho(v)|-\ell(\rho(v))-1$ for some $v\in V_0$.  Therefore $p_*(\xi^{D}\zeta_{\Gamma,b\, *}[\PP\oH(\Gamma,b,\rho)])$  vanishes unless 	
	$$
	D\leq |V_0|-1 \sum_{v\in V_0} 2g(v)+|\rho(v)|-\ell(\rho(v))-1.
	$$ 
	First, we remark that $$\sum_{V_0} |\rho(v)|-\ell(\rho(v))\leq |\rho|-\ell(\rho),$$ and the equality holds only if and for all $1\leq k\leq \ell(\rho)$ we have: the legs with label $i$ satisfying $$n+\rho_1+\ldots + \rho_{k-1} < i \leq  n+\ldots + \rho_k$$ are incident to the same vertex in $V_0$. Besides, if we assume that  last leg is incident to a vertex in $V_{-1}$, then 
	\begin{equation}\label{for:ineqres} \sum_{V_0} |\rho(v)|-\ell(\rho(v))\leq |\rho|-\ell(\rho)-1,
	\end{equation} and the equality holds if the above condition holds for $1\leq k<\ell(\rho)$, but if $k=\ell(\rho)$  then either $\rho_k=2$ and the last two legs are incident to vertices in $V_{-1}$, or otherwise the last leg sits on a vertex in $V_{-1}$ while the legs with label $i$ satisfying
 $$n+\rho_1+\ldots + \rho_{\ell-1} < i <  n+m$$ are incident to the same vertex in $V_0$. Then $p_*(\xi^{d}\zeta_{\Gamma,b\, *}[\PP\oH(\Gamma,b,\rho)])$ vanishes unless
 \begin{eqnarray*}
 \left(2g -\sum_{v\in V_0} 2g(v)\right) + \left( |\rho|-\ell(\rho) - \sum_{v\in V_0}  |\rho(v)|-\ell(\rho(v))\right)\leq 0.
 \end{eqnarray*} This inequality imposes that $\sum_{v} 2g(v) = g$, i.e. the graph is of compact type and all vertices are of genus 0, and that we have the equality in~\eqref{for:ineqres}. This second condition implies that the residue condition is trivial and there can be only one vertex in $V_{-1}$, hence $(\Gamma,b)$ is a simple back-bone graph. The same reasoning finishes the proof of the second part of the lemma.	\end{proof}

\subsection{Description of $\oM_{-1}$} We fix a back-bone graph $(\Gamma,b)$ compatible with $(\mu,\rho)$, and we denote by $v_{-1}$ the unique vertex in $V_{-1}$. The restriction $b(v_{-1})$ of the twist $b$ at this vertex defines a vector in $\ZZ^{n(v_{-1})}$. Then we denote by $\H_{0,n_{-1}}(b(v_{-1}))$ the moduli space of marked meromorphic differentials $(C,\eta,x_1,\ldots)$ satisfying: 
\begin{itemize}
\item the poles of $\eta$ are markings;
\item ${\rm ord}_{x_i}(\eta)\geq b(i)-1$ for all $1\leq i\leq n(v_{-1})$. 
\end{itemize}
\sskip

 If $(\Gamma,b)$ is a simple back-bone graph, then we denote by $\H(\Gamma,b,\rho)_{-1}$,  the subspace of $\H_{0,n_{-1}}(b(v_{-1}))$ of differentials with trivial residues at the poles. If $(\Gamma,b)$ is an almost simple back-bone graph, then we denote by $\H(\Gamma,b,\rho)_{-1}$,  the subspace of $\H_{0,n_{-1}}(b(v_{-1}))$ of differentials with no trivial residues at the poles apart: the last two legs if the last entry of $\rho$ is $2$, or the last leg and the half-edge connecting to the vertex carrying the legs with label in $\{n+m-\rho_{\ell},\ldots,n+m-1\}$.

\begin{definition} The space $\oM_{-1}$ is the closure of the image $\PP\H(\Gamma,b,\rho)_{-1}$ in $\oM_{0,n_{-1}}$ under the map forgetting the differential form. Moreover, we denote 
$$
{[}\PP\oH(\Gamma,b,\rho){]}= \left\{
\begin{array}{cl}
\text{P.D. class of $\PP\oH(\Gamma,b,\rho),$} & \text{if ${\rm dim}\oH(\Gamma,b,\rho)={\rm dim} \oH_g(\mu,\rho)$}\\
0, & \text{otherwise.}
\end{array}
\right.
$$
This class sits in $H^*(\oM_{0,n(v_{-1})}) \otimes H^*(\PP\bigotimes_{v\in V_0} \oH_{g(v),n(v),m(v)})$. 
\end{definition}

\subsection{Induction formulas in cohomology}\label{ssec:cohomology}

If we choose $i\in \{1,\ldots,n+m\}$, then we denote respectively by  ${\rm Bic}(\mu,\rho)_i$ and ${\rm SBB}(\mu,\rho)_i$ the sets of bi-colored graphs and complete simple back-bone graphs such that the $i$-th leg is incident to the unique vertex in $V_{-1}$.  
Besides, we denote by ${\rm ASBB}(\mu,\rho)$ the set of complete almost simple  graphs. The following proposition recalls the formulas in cohomology that will be used in the rest of the text. 
\begin{proposition}[\cite{SauGT}, Theorem~5] 
If we assume that  $(\mu,\rho)$ is complete then we have 
\begin{eqnarray}\label{for:ind1}
	(\xi + k_i\psi_i) [\PP\oH_g(\mu,\rho)] &=& \sum_{(\Gamma,b)\in {\rm SBB}(\mu,\rho)_i} \frac{m(\Gamma,b)}{|{\rm Aut}(\Gamma,b)|}\zeta_{\Gamma,b *}[\PP\oH(\Gamma,b,\rho)] + \Delta, \\ \label{for:ind2}
	\xi  [\PP\oH_g(\mu,\rho)] &=& \sum_{(\Gamma,b)\in {\rm ASBB}(\mu,\rho)_i} \frac{m(\Gamma,b)}{|{\rm Aut}(\Gamma,b)|}\zeta_{\Gamma,b *}[\PP\oH(\Gamma,b,\rho)] + \Delta'.
\end{eqnarray}
Here $\Delta$ and $\Delta'$ are cohomology classes supported on the union of boundary components associated with graphs in ${\rm Bic}(\mu,\rho)_i\setminus {\rm SBB}(\mu,\rho)_i$ and ${\rm Bic}(\mu,\rho)\setminus  {\rm ASBB}(\mu,\rho)$ respectively. If $\mu$ is not complete then 
\begin{eqnarray} \label{for:ind3}
	(\xi + k_i\psi_i) [\PP\oH_g(\mu,\rho)] &=& [\PP\oH_g(\mu',\rho)] \\ &&+ \sum_{(\Gamma,b)\in {\rm SBB}(\mu,\rho)_i} \frac{m(\Gamma,b)}{|{\rm Aut}(\Gamma,b)|}\zeta_{\Gamma,b *}[\PP\oH(\Gamma,b,\rho)] + \Delta \end{eqnarray}
	where $\mu'$ is the partition obtained from $\mu$ by adding 1 to the $i$th coordinate, while $\Delta$ is a cohomology class supported on the union of boundary components associated with graphs in ${\rm Bic}(\mu,\rho)_i\setminus {\rm SBB}(\mu,\rho)_i$.
\end{proposition}

We first use these formulas to complete the proof of Proposition~\ref{prop:independence} in the introduction
\begin{proof} 
We write $(m_1\psi_1-m_2\psi_2)[\PP\oH_g(\mu,\rho)]$ using the differences between~\eqref{for:ind1} for $i=1$ and $2$: 
  \begin{eqnarray*}(m_1\psi_1-m_2\psi_2) [\PP\oH_g(\mu,\rho)] &=& \!\!\!\!\! \sum_{(\Gamma,b)\in {\rm SBB}(\mu,\rho)_1} \frac{m(\Gamma,b)}{|{\rm Aut}(\Gamma,b)|}\zeta_{\Gamma,b *}[\PP\oH(\Gamma,b,\rho)]  \\ &&- \!\!\!\!\!\!\!\!\!\!\sum_{(\Gamma,b)\in {\rm SBB}(\mu,\rho)_2} \frac{m(\Gamma,b)}{|{\rm Aut}(\Gamma,b)|}\zeta_{\Gamma,b *}[\PP\oH(\Gamma,b,\rho)] + \Delta.
  \end{eqnarray*}
  We multiply this identity by $\alpha=\xi^{d(\mu,\rho)-n+1}\prod_{i=3}^n k_i\psi_i$ and integrate. The left-hand-side is given by $a(\mu,\rho)_2-a(\mu,\rho)_1$. For the right-hand-side, we use the first part of Lemma~\ref{lem:vanishsbb}, we see that all graphs outside $ {\rm SBB}(\mu,\rho)_1\cup$ and $ {\rm SBB}(\mu,\rho)_2$ have trivial contribution. Moreover, if $(\Gamma,b)$ is a graph in ${\rm SBB}(\mu,\rho)_1\setminus {\rm SBB}(\mu,\rho)_2$ then the second leg is incident to a vertex in $V_0$ and the integral $\alpha$ vanishes for dimension reasons. 
\end{proof}

\section{Induction on the genus for minimal strata}

 In this section we assume that $n=1$. In order to establish Theorem~\ref{thm1} we first reformulate it by extracting a coefficient in $t$-variables in~\eqref{for:thm1}. We fix a half-integer $g$ and $\rho$, and we write $g'=g-|\rho|/2$. We assume that $\rho$ is in the form of a partition i.e. $\rho_1\geq \rho_2\ldots\geq \rho_\ell$. 
Then  we extract the coefficient of $t_{\rho_1}\ldots t_{\rho_\ell}$ of~\eqref{for:thm1}. It gives the equality
\begin{equation}\label{for:thm1b}
\frac{(2g)! b_{g'}}{|{\rm Aut}({\rho})|} = \sum_{r\geq 1}   \binom{2g}{r}  \underset{\rho^1+\ldots + \rho^r ={\rho}}{\sum_{k_1+\ldots+k_r=2g-r}}\prod_{i=1}^r \frac{a((k_i) ,\rho^{i})}{|{\rm Aut}(\rho^{i})|}.
\end{equation}
In this expression, the sum is over all $r$-uple of possibly trivial partitions with concatenation (as partitions) equal to ${\rho}$, while ${\rm Aut}({\rho})$ is the automorphism group of the partition.  This is the equality that we are going to prove in this section. we need to remark that 
\begin{eqnarray}
(2g)! b_{g'}&=& 2g(2g-1) \int_{\oM_{g',1+|\rho|}}\!\!\!\!\!\! (-1)^{g'} \lambda_{g'} \prod_{j=1}^{2g-2} j \psi_1 \nonumber \\ \label{for:b} &=& 2g(2g-1) \int_{\PP\oH_{g',1,\rho}} \!\!\!\!\!\! \xi^{2g'-1+|\rho|-\ell} \prod_{j=1}^{2g-2} j \psi_1.
\end{eqnarray}
Besides we remark that, applying the identity~\eqref{for:ind3} $(2g-2)$ times, we obtain an expression of the form
$$
\prod_{j=1}^{2g-1} j \psi_1 = [\PP\oH_{g'}((2g-1),\rho)] + \text{boundary contribution} + O(\xi)
$$
where $O(\xi)=\xi\cdot C$ for  some class $C$ in $H^*(\PP\oH_{g',1,\rho})$. So, once we multiply this expression by $\xi^{2g-1-\ell}$ and integrate, the class $\xi\cdot C$ is sent to 0 by Lemma~\ref{lem:vanishsbb}. Therefore we obtain 
 $$
 (2g-2)! b_{g'} = a((2g-1),\rho) + \text{boundary contribution} 
 $$
and we are left to analyze this boundary contribution. To do so we apply~\eqref{for:ind3} in the case $n=1$ and $\mu=(k)$ for $k<2g-1$ to obtain
\begin{eqnarray*}
k\psi_1[\PP\oH_{g'}((k),\rho)] &=&[\PP\oH_{g'}((k+1),\rho)] \\
&&+\!\!\!\!\!\! \sum_{(\Gamma,b)\in {\rm SBB}((k),\rho)_1}  \frac{m(\Gamma,b)}{|{\rm Aut}(\Gamma,b)|}\zeta_{\Gamma,b *}[\PP\oH(\Gamma,b,\rho)] + \Delta + O(\xi),
\end{eqnarray*}
where the class $\Delta$ is supported on the boundary components of $[\PP\oH_{g'}((k),\rho)]$ defined by graphs in ${\rm Bic}((k),\rho)_1\setminus {\rm SBB}((k),\rho)_1$. If we multiply this expression by $\xi^{2g-1-\ell}\psi_1^{2g-2-k}$ and integrate then the contribution of $\Delta$ and $O(\xi)$ are trivial by Lemma~\ref{lem:vanish} and Lemma~\ref{lem:vanishsbb} respectively. 

\sskip

In order to obtain the equality~\eqref{for:thm1b}, we consider the action of ${\rm Aut}(\rho)$ on  $\oM_{g',1+|\rho|}$ by permutation of markings as follows: if $\sigma\in {\rm Aut}(\rho)\subset S_\ell$, then for all $1\leq i\leq \ell$ and $1\leq j\leq \rho_{i}$ we set 
\begin{equation*}
\sigma\left(1+\left(\sum_{1\leq i'<i} \rho_{i'}\right)+j\right) = 1+\left(\sum_{1\leq i'<\sigma{i}} \rho_{\sigma^{-1}(i')}\right)+j
\end{equation*} 
In other words, the action of ${\rm Aut}(\rho)$ permutes sets of markings tied by residue condition  without changing the order within this set. For all $k$, the locus $\PP\oH_{g'}(k,\rho)$ is invariant under the action of ${\rm Aut}(\rho)$, so the identity
\begin{eqnarray}
\nonumber {(2g-2)!}{\psi_1^{2g-2}} &=& [\PP\oH_{g'}((2g-1),\rho)] \\ &&+ \sum_{k=1}^{2g-2} \sum_{(\Gamma,b)\in {\rm SBB}((k),\rho)_1}  \frac{m(\Gamma,b)}{|{\rm Aut}(\Gamma,b)|}\zeta_{\Gamma,b *}[\PP\oH(\Gamma,b,\rho)] + \Delta + O(\xi),
\end{eqnarray}
is valid in $H^*(\PP\oH_{g',1+|\rho|}/{\rm Aut}(\rho))$. 
\sskip

Next, for a fixed graph $(\Gamma,b)$ in ${\rm SBB}((k),\rho)_1$,  we will compute the integral 
\begin{eqnarray*}
\int_{\PP\oH(\Gamma,b,\rho)}\!\!\!\!\!\! \xi^{2g-1-\ell}\psi_1^{2g-2-k} &=& \left(  \int_{\oM_{-1}}\!\!\!\!\!\! \psi_1^{2g-2-k}\right)\times \int_{\PP\prod_{v\in V_0} \oH_{g(v)}(\mu(v),\rho(v))} \!\!\!\!\!\!\!\!\!\!\!\!\!\!\! \xi^{2g-1-\ell}. \end{eqnarray*}
By the same argument as in~\cite[Lemma~3.7]{SauGAFA} we have 
$$
\int_{\oM_{-1}} \psi_1^{2g-2-k} = \left\{ \begin{array}{cl} 1 & \text{if $2g-k=|V_0|$},\\ 0 & \text{otherwise.} \end{array}\right.
$$
On the other hand, the definition of ${\rm SBB}((k),\rho)_1$ implies that the vector  $\mu(v)$ is equal to $(2g(v)+|\rho(v)|-1)$  for all $v\in V_0$. Therefore, we obtain the following expression
$$
 \int_{\PP\prod_{v\in V_0} \oH_{g(v)}(\mu(v),\rho(v))} \xi^{2g-1-\ell} = \prod_{v\in V_0} a((2g(v)+|\rho(v)|-1),\rho(v)).
$$
Putting everything together we obtain the following relation
\begin{eqnarray*}
\frac{(2g)!}{|{\rm Aut}(\rho)|} b_{g'} &=& 2g (2g-1)a((2g-1),\rho) \\ &&+ \sum_{k=1}^{2g-2}  \sum_{(\Gamma,b) \in {\rm SBB}((k),\rho)_1} \frac{2g!}{k!}  \frac{m(\Gamma,b)}{|{\rm Aut}(\Gamma,b)||{\rm Aut}(\rho)|}   \\ && \,\,\,\,\,\,\,\,\,\,\,\,\,\,\,\,\,\,\,\,\,\,\,\,\,\,\,\,\,\,\,\,\,\,\,\,\,\,\,\,\,\,\,\,\,\,\,\,\,\,\,\,\,\,\,\,\,\,\,\,\,\,\,\,\,\,\,\,\,\,\,\,\,\,\,\,\,\,\,\,\int_{\PP\oH(\Gamma,b,\rho)} \!\!\!\!\!\! \xi^{2g-1-\ell}\psi_1^{2g-2-k}  \\
&=& 2g (2g-1) a((2g-1),\rho) \\ &&+ \sum_{k=1}^{2g-2}  \underset{|V_0|=2g-k}{\sum_{(\Gamma,b) \in \widetilde{\rm SBB}((k),\rho)_1}} \frac{(2g)!}{k!}   \frac{1}{|\widetilde{\rm Aut}(\Gamma,b)|} \prod_{v\in V_0} \frac{k(v) a((k(v)),\rho(v))}{|{\rm Aut}(\rho(v)|}.
\end{eqnarray*}
In this expression $k(v)=2g(v)+|\rho(v)|-1$. Besides,  $\widetilde{\rm Aut}(\Gamma,b)$ is the extension of ${\rm Aut}(\Gamma,b)$ obtained by allowing permutation of the markings in ${\rm Aut}(\rho)$, while $\widetilde{\rm SBB}((k),\rho)_1$ is the set of isomorphism classes of such stable graphs up to the action of ${\rm Aut}(\rho)$. Then this formula can be rewritten as
\begin{eqnarray*}
\frac{(2g)!}{|{\rm Aut}(\rho)|} b_{g'} &=& 2g (2g-1)a((2g-1),\rho)\\&& + \sum_{k=1}^{2g-2}\frac{(2g)!}{k!(2g-k)!}  \underset{\rho^1+\ldots+\rho^{2g-k}=\rho}{\sum_{k_1+\ldots+k_{2g-k} = k }}    \prod_{i=1}^{2g-k} k_i \frac{a((k_i),\rho^i)}{|{\rm Aut}(\rho^i)|}\\
&=&  \sum_{r\geq 1}   \binom{2g}{r}  \underset{\rho^1+\ldots + \rho^r ={\rho}}{\sum_{k_1+\ldots+k_r=2g-r}}\prod_{i=1}^r \frac{a((k_i) ,\rho^{i})}{|{\rm Aut}(\rho^{i})|}
\end{eqnarray*}
which is the desired~\eqref{for:thm1b} (in the last sum we have re-ordered the summation with $r=2g-k$).

\section{Volume recursion for the numbers of zeros}

 In this section, we establish Theorem~\ref{thm2} by following the strategy of~\cite{CMSZ}. We fix a complete pair $(\mu,\rho)$ and we assume that $n\geq 2$. We will give a special role to the legs $1$ and $2$. First we use the identity~\eqref{for:ind1} for $i=2$
  \begin{equation}
	k_2\psi_2 [\PP\oH_g(\mu,\rho)] = \sum_{(\Gamma,b)\in {\rm SBB}(\mu,\rho)_{2}}  \frac{m(\Gamma,b)}{|{\rm Aut}(\Gamma,b)|} \zeta_{\Gamma,b*} [\PP\oH(\Gamma,b,\rho)] + \Delta + O(\xi) 
\end{equation}
where $\Delta$ is a sum of contributions supported on strata associated with graphs in $ {\rm Bic}(\mu,\rho)_{2}\setminus  {\rm SBB}(\mu,\rho)_{2}$ and $O(\xi)$ is of the form $\xi C$ for $C\in H^*(\PP\oH_{g,n,\rho})$. We multiply this identity by $\xi^{d(\mu,\rho)-n}\prod_{i=3}^n k_i\psi_i$ and integrate. In the right-hand-side, both the contribution of $\Delta$ and $O(\xi)$ vanish by Lemma~\ref{lem:vanish} and Lemma~\ref{lem:vanishsbb}. We now analyze the contribution of back-bone graphs:
\begin{eqnarray*}
	\int_{\PP\oH(\Gamma,b,\rho)}\!\!\!\!\!\! \xi^{d(\mu,\rho)-n} \prod_{i=3}^n \psi_i &=& \left(\int_{\oM_{-1}} \underset{i\geq 3}{\prod_{i\mapsto V_{-1}}} k_i\psi_i\right) \\ && \,\,\,\,\,\,\,\,\,\times \int_{\PP\prod_{v\in V_0} \oH_{g(v)}(\mu(v),\rho(v))} \!\!\!\!\!\!\!\!\xi^{d(\mu,\rho)-n} \underset{i\geq 3}{\prod_{i\mapsto  V_0}} k_i \psi_i
\end{eqnarray*}
For dimension reasons, the first term in this product vanishes unless the first leg is incident to the vertex in $V_{-1}$. Then this integral is equal to 
$$
\widetilde{h}_{\PP^1}\big((k_i)_{i\mapsto V_{-1}}, (b(e))_{e\in E(\Gamma)}\big) \times \prod_{v\in V_0} a(\mu(v),\rho(v))  
$$
where the function $\widetilde{h}_{\PP^1}$ is close to the the function $h_{\PP^1}$ from~\cite{CMSZ}\footnote{We have $\widetilde{h}_{\PP^1}=h_{\PP^1} \times \prod_{i\geq 3} k_i$} and defined as follows. If $\mu[0]=(k_1,\ldots,k_n),$ and $\mu[\infty]=(p_1,\ldots,p_m))$ are vectors of positive integers satisfying $$|\mu[0]|-|\mu[\infty]|=-2+\ell(\mu[0])+\ell(\mu[\infty]),$$ and such that $\ell(\mu[0])\geq 2$, then $\overline{H}_{\PP^1}(\mu[0],\mu[\infty])$ is the compactification of the space of differentials with no residues on a genus 0 surface with zeros of orders $k_1-1,k_2-1,\ldots$ and poles of order $p_i+1,\ldots,$ and 
$$
\widetilde{h}_{\PP^1}(\mu[0],\mu[\infty])=\int_{\PP\overline{H}_{\PP^1}(\mu[0],\mu[\infty])} \prod_{i=3}^n  k_i\psi_i.$$
Then with this function we obtain the following recursion
\begin{equation}\label{for:recursionn}
	a(\mu,\rho)= \sum_{(\Gamma,b)\in {\rm SBB}(\mu,\rho)_{1,2}} \frac{m(\Gamma,b)}{|{\rm Aut}(\Gamma,b)|}  h_{\PP^1}((k_i)_{i\mapsto V_{-1}}, (b(e))_{e\in E}) \prod_{v\in V_0} a(\mu(v),\rho(v)),
\end{equation}
where ${\rm SBB}(\mu,\rho)_{1,2}$ is the set of simple back-bone graphs such that the legs 1 and 2 are incident to the vertex in $V_{-1}$. Then we can rewrite this equality as follows
\begin{eqnarray}
	\nonumber \frac{a(\mu,\rho)}{|{\rm Aut}(\rho)|}= \sum_{r\geq 1} \!\!\!\!\!\!\!\!\! \underset{K_1+\ldots+K_r=|\mu_{L_{-1}}|+k_1+k_2-|L_{-1}|-r }{\sum_{L_{-1}\sqcup L_1 \sqcup L_r=\{3,\ldots,n\}}} && \!\!\!\!\!\!\!\!\!\!\!\!\!\!\!\!\!\!\frac{\widetilde{h}_{\PP^1}((k_1,k_2)+\mu^{-1},(K_1,\ldots,K_r))}{r!} \\ &&\!\!\!\!\!\!\!\! \!\!\!\!\!\!\!\!\times \left(\sum_{\rho^1+\ldots + \rho^r =\rho} \prod_{j=1}^r \frac{K_j a((K_j)+\mu^j,\rho^j)}{|{\rm Aut}(\rho^j)|}\right).
\end{eqnarray}
In this equality $\mu^j$ is the vector of positive integers obtained from $\mu$ by restriction to the subset $L_j$ and the notation $(K_j)+\mu^j$ stands for the concatenation of vectors, while the sum is over choices of $r$ partitions that concatenate as partitions to $\rho$. Now we can use this identity to evaluate $\a(\mu)=\sum_{\rho} \frac{a(\mu,\rho)}{|{\rm Aut}(\rho)|} \prod t_{\rho_i}$. We obtain
\begin{eqnarray}
	\nonumber \a(\mu)= \sum_{r\geq 1} \!\!\!\!\!\!\!\!\! \underset{K_1+\ldots+K_r=|\mu^{{-1}}|+k_1+k_2-|L_{-1}|-r }{\sum_{L_{-1}\sqcup L_1 \sqcup L_r=\{3,\ldots,n\}}} && \!\!\!\!\!\!\!\!\!\!\!\!\!\!\!\!\!\!\frac{\widetilde{h}_{\PP^1}((k_1,k_2)+\mu^{-1},(K_1,\ldots,K_r))}{r!} \\ &&\!\!\!\!\!\!\!\! \!\!\!\!\!\!\!\!\times \left(\prod_{j=1}^r K_j \a((K_j)+\mu^j,\rho^j)\right).
\end{eqnarray}

This recursion formula is equivalent to~\cite{CMSZ}[Lemma 3.13], with the only difference that we replaced the rational numbers $a(\mu,\emptyset)$ by the full polynomials $\a(\mu)$ in $t$-variables. However, all subsequent re-organization of this recursion formulas are valid without modification, in particular the one stated in the intro as Theorem~\ref{thm2}.

\section{Integrable hierarchy for minimal strata}

 Here we prove the last recursive formula provided by Theorem~\ref{thm3}. To do so we fix some vectors $\mu$ and $\rho$ as in the introduction and we assume that $\ell(\rho)\geq 1$. Besides, we write $\rho'=(\rho_1,\ldots,\ldots,\rho_{\ell-1})$, and for simplicity we assume that we write it as a partition, i.e. $\rho_1\geq \ldots\geq \rho_{\ell-1}$ (but we do not make any assumption on the last entry $\rho_\ell$). 
 \sskip
 
The first step is to multiply~\eqref{for:ind2} by $\xi^{d(\mu,\rho)-1-n}\prod_{i= 2}^n$ and integrate to obtain
\begin{eqnarray*}
a(\mu,\rho) &=& \sum_{(\Gamma,b)\in {\rm ASBB}(\mu,\rho)} \frac{m(\Gamma,b)}{|{\rm Aut}(\Gamma,b)|} \int_{\PP\oH(\Gamma,b,\rho)} \xi^{d(\mu,\rho)-1-n}\prod_{i\geq 3}^n k_i\psi_i.
\end{eqnarray*}
Here we have used the fact that the contribution of bi-colored graphs which are not in ${\rm ASBB}(\mu,\rho)$ is trivial by Lemma~\ref{lem:vanishsbb}. Besides, the contribution in $ {\rm ASBB}(\mu,\rho)$ vanishes unless the first leg is incident to the vertex in $V_{-1}$ for dimension reasons. Therefore, we denote by ${\rm ASBB}(\mu,\rho)_1\subset {\rm ASBB}(\mu,\rho)$ the set of graphs satisfying this constraint.  If we analyze further the contribution of graphs in  ${\rm ASBB}(\mu,\rho)_1$ we obtain
\begin{equation}\label{for:indres}
a(\mu,\rho) = \!\!\!\!\!\!\!\!\!\!\sum_{(\Gamma,b)\in {\rm ASBB}(\mu,\rho)_1} \!\!\!\!\!\frac{m(\Gamma,b)}{|{\rm Aut}(\Gamma,b)|} \widehat{h}_{\PP^1}((k_i)_{i\mapsto V_{-1}}, \mu[\infty]) \prod_{v\in V_0} a(\mu(v),\rho(v)).
\end{equation}
In this expression, $\mu[\infty]$ is either $(b(e))_{e\in E}+(0,0)$ if $\rho_\ell=2$, and otherwise it is given by $$(b(e))_{e\in E\setminus \{e_0\}}+(b(e_0),0)$$ where $e_0$ is the edge connecting the vertex in $V_{-1}$ to the vertex in $V_0$ carrying the legs between $n+m-\rho_\ell+1$ and $n+m-1$.  This expression is similar to~\eqref{for:recursionn}. Apart from the set of indexation of the sum,  the difference is the replacement of $\widetilde{h}_{\PP^1}$ by $\widehat{h}_{\PP^1}$. The definition of this last function is 
\begin{equation}\widehat{h}_{\PP^1}((k_1,\ldots,k_n),(p_1,\ldots,p_m))=\int_{\PP\widehat{H}_{\PP^1}(\mu[0],\mu[\infty])} \prod_{i=2}^n k_i\psi_i,
\end{equation}
where $\widehat{H}_{\PP^1}(\mu[0],\mu[\infty])$ is the space of differentials on a genus 0 curve with zeros and poles $\mu[0]$ and $\mu[\infty]$ and with vanishing residues appart from the residues at the  last last two poles. 
\begin{remark} Formula~\eqref{for:indres} expresses the Masur--Veech volume in terms of the divisor indexed by $(\Gamma,b) \in {\rm ASBB}(\mu,\rho)_1$. We see that this choice depends on:
\begin{itemize}
\item  the choice of an element $i\in \{1,\ldots,n\}$ (here we have implicitly fixed $i=1$);
\item an ordering of the poles. 
\end{itemize}
\end{remark}
\begin{remark} For $n>1$,  formula~\eqref{for:indres} involves graphs such that $\M_{-1}$ is of positive dimension, while it would be desirable to have an alternative expression on boundary components with $\oM_{-1}$ of dimension 0. For $\rho=(2)$, these graphs index cylinder configurations, and the re-organization of~\eqref{for:indres} into an expression involving only cylinder configurations is provided by~\cite{IttSau}. 
\end{remark}

From now on, we impose $n=1$. Then the function $\widehat{h}_{\PP^1}$ has a simple expression found in~\cite{CosSauSch} (and by alternative method in~\cite{BurRos}), i.e.
$$
\widehat{h}_{\PP^1}((k_1),(p_1,\ldots,p_m))= (m-2)! \prod_{i=1}^{m-2} p_i
$$
Using this expression we will re-write formula~\eqref{for:indres}. We separate the cases.

\subsubsection*{Case $\rho_\ell=2$} Under this assumption the vector $\mu[\infty]$ in~\eqref{for:indres} is $(b(e))_{e\in E}+(0,0)$, and we re-write:
\begin{eqnarray*}
a((k),\rho) &=& \!\!\!\!\!\!\!\!\!\!\sum_{(\Gamma,b)\in {\rm ASBB}(\mu,\rho)_1} \!\!\!\!\!\frac{m(\Gamma,b)^2\times |E(\Gamma)|!}{|{\rm Aut}(\Gamma,b)|}  \prod_{v\in V_0} a(\mu(v),\rho(v)).
\end{eqnarray*}
Then we can re-write this expression as
\begin{eqnarray*}
\frac{a((k),\rho)}{|{\rm Aut}(\rho')|} &=& \sum_{r\geq 1}  \underset{\rho^1+\ldots + \rho^r ={\rho'}}{\sum_{k_1+\ldots+k_r=k+1-r}} \prod_{i=1}^r \frac{k_i^2 a((k_i) ,\rho^{i})}{|{\rm Aut}(\rho^{i})|},
\end{eqnarray*}
This expression for a single choice of $\rho'$. As we sum over all  choices of $\rho'$, we obtain an expression of $\frac{\partial}{\partial t_2}\a((k))=\sum_{\rho'} \frac{a((k),\rho'+(2))}{|{\rm Aut}(\rho')|}$:
\begin{eqnarray*}
\frac{\partial}{\partial t_2}\a((k)) &=& \sum_{r\geq 1}  {\sum_{k_1+\ldots+k_r=k-r-1}} \prod_{i=1}^r k_i^2\a((k_i)),
\end{eqnarray*}
which is equivalent to the PDE~\eqref{for:thm31}. 

\subsubsection*{Case $\rho_\ell>2$} For this second case, the vector $\mu[\infty]$ in~\eqref{for:indres} is the vector of integers $(b(e))_{e\in E\setminus\{e_0\}}+(b(e_0),0)$, and one vertex $v_0\in V_0$ is distinguished from the others as it carries the last legs. Then we have the following recursion:
\begin{eqnarray*}
\frac{a((k),\rho)}{|{\rm Aut}(\rho')|} &=& \sum_{r\geq 0}  \underset{\rho^0+\ldots + \rho^r ={\rho'}}{\sum_{k_0+k_1+\ldots+k_r=k-r}} \frac{k_0 a((k_0) ,\rho^{0}+(\rho_\ell-1))}{|{\rm Aut}(\rho^{0})|} \prod_{i=1}^r \frac{k_i^2 a((k_i) ,\rho_{i})}{|{\rm Aut}(\rho^{i})|},
\end{eqnarray*}
For a fixed value of $\rho_\ell$, if we sum over all choices of $\rho'$, then we obtain
\begin{eqnarray*}
\frac{\partial}{\partial t_{\rho_\ell}}\a((k)) &=& \sum_{r\geq 0}  {\sum_{k_0+k_1+\ldots+k_r=k-r}} k_0 \frac{\partial}{\partial t_{\rho_\ell}}\a((k_0)) \prod_{i=1}^r k_i^2 \a((k_i)),
\end{eqnarray*}
which is equivalent to the PDE~\eqref{for:thm32}.

\newcommand{\etalchar}[1]{$^{#1}$}

\end{document}